\documentclass[a4paper,11pt,fleqn]{article}
\usepackage[english]{babel}
\usepackage{amsfonts}
\usepackage{amsthm}
\usepackage{amsbsy}
\usepackage{amssymb}
\usepackage{graphicx}
\usepackage{eso-pic}
\usepackage{dsfont}
\usepackage{tikz}
\usepackage{xfrac}
\usepackage{float}
\usepackage[bookmarks=true]{hyperref}
\usepackage{bookmark}
\usepackage{amsmath}
\usetikzlibrary{matrix,arrows}
\theoremstyle{definition}
\newtheorem{defin}{Definition}[section]
\newtheorem{oss}[defin]{Remark}
\newtheorem{prop}[defin]{Proposition}
\newtheorem{teo}[defin]{Theorem}

\newcommand\simtimes{\mathbin{%
    \stackrel{\sim}{\smash{\times}\rule{0pt}{0.6ex}}%
    }}

\begin{document}

\begin{center}
{\large \textbf{SOME HYPERBOLIC 4-MANIFOLDS WITH LOW VOLUME AND NUMBER OF CUSPS}

\vspace{0.4 in}}

\textnormal{LEONE SLAVICH}

leone.slavich@gmail.com
\end{center}

\noindent ABSTRACT: We construct here two new examples of non-orientable, non-compact, hyperbolic $4$-manifolds. The first has minimal volume $v_m=4\pi^2/3$ and two cusps. This example has the lowest number of cusps among known minimal volume hyperbolic $4$-manifolds. 
The second has volume $2\cdot v_m$ and one cusp. It has lowest volume among known one-cusped hyperbolic $4$-manifolds.

\section{Introduction}
As proven by Ratcliffe and Tschanz in \cite{ratcliffetschanz2}, the volume spectrum of hyperbolic $4$-manifolds is the set of integral multiples of 
$v_m=4\pi^2/3$. Regarding volume as a measure of the complexity of a hyperbolic manifold, it is desirable to have as many examples as possible of minimal volume hyperbolic $4$-manifolds. 

The largest known census of minimal-volume cusped hyperbolic $4$-manifolds is contained in \cite{ratcliffetschanz2}, and consists of $1171$ distinct manifolds, all obtained by pairing the facets of one copy of the regular ideal hyperbolic $24$-cell. All these manifolds have five or six cusps.
In section \ref{2cuspidi}, we will prove the following:
\begin{teo}\label{minvol}
There exists a minimal volume, non-orientable, hyperbolic $4$-manifold with two cusps. \end{teo} 
This manifold will be constructed by pairing the facets of a regular ideal hyperbolic $24$-cell. To the best of our knowledge, the manifold of Theorem \ref{minvol} has the lowest number of cusps among known minimal volume hyperbolic $4$-manifolds.

Kolpakov and Martelli \cite{martelli} have recently constructed the first example of hyperbolic $4$-manifold with one cusp. The volume of their example is $4\cdot v_m$.
In section \ref{1cuspide}, we will prove the following:
\begin{teo}\label{2minvol} There exists a one-cusped, non-orientable, hyperbolic $4$-manifold with volume equal to $2\cdot v_m$.\end{teo} 
This example will be obtained by pairing the facets of two copies of the regular ideal hyperbolic $24$-cell.
To the best of our knowledge, this manifold of Theorem \ref{2minvol} has the lowest volume among known one-cusped hyperbolic $4$-manifolds.
%Its orientable double cover is the one-cusped hyperbolic $4$-manifold constructed in \cite{martelli}.
\section{The regular ideal hyperbolic 24-cell}

We begin by describing the hyperbolic ideal polytope which we use to build our manifolds.

The 24-cell $C$ is the only regular polytope in all dimensions $n\geq 3$ which is self-dual and not a simplex. It is defined as the convex hull in $\mathbb{R}^4$ of the set of 24 points obtained by permuting the coordinates of $$(\pm 1, \pm 1, 0 , 0 ).$$ 
It has $24$ vertices, $96$ edges, $96$ faces and $24$ $3$-dimensional facets which lie in the affine hyperplanes of equations

$$ x_i=\pm 1, \;\;\; \pm x_1 \pm x_2 \pm x_3 \pm x_4 =2. $$
The dual polytope $C^*$ is the convex hull $$C^*=\text{Conv}\{\mathcal{R}, \mathcal{B}, \mathcal{G}\}$$ where
$\mathcal{G}$ is the set of $8$ points obtained by permuting the coordinates of $$(\pm 1,0,0,0)$$ and
$ \mathcal{R}\cup \mathcal{B}$ is the set of $16$ points of the form
$$\left(\pm \frac{1}{2}, \pm \frac{1}{2},\pm \frac{1}{2},\pm \frac{1}{2}\right)$$ with $\mathcal{R}$ (resp. $\mathcal{B}$) being the set of $8$ points with an even (resp. odd) number of minus signs. The facets of $C$ are regular octahedra in canonical one-to-one correspondence with the vertices of $C^*$, and are coloured accordingly in red, green and blue. This coloring is natural, every symmetry of the $24$-cell $C$ preserves the partition of the vertices of $C^*$ into the sets $\mathcal{R}, \mathcal{B}, \mathcal{G}$, and every permutation of $\{\mathcal{R},\mathcal{G},\mathcal{B}\}$ is realized by a symmetry of $C$.
The vertex figure is a cube, in accordance with self-duality.
Notice furthermore that the convex envelope of $\mathcal{R}\cup \mathcal{B}$ is a hypercube, while the convex envelope of $\mathcal{G}$ is a $16$-cell. 

Each octahedral facet of the $24$-cell possesses a checkerboard coloring on its triangular $2$-faces. Consider for instance a red facet $F$ of the $24$-cell $C$. Each triangular face $f$ of $F$ will be adjacent to $F$ and another octahedral facet $G$ which will be colored either in green or blue. We assign to the face $f$ the color of $G$. This induces a green/blue coloring on the triangular faces of $F$, such that if a face $f$ is adjacent to a face $g$, their colors are different. Notice that green facets will be endowed with a red/blue coloring, while the blue facets will be endowed with a red/green coloring.

%Every $n$-dimensional regular polytope has a hyperbolic idel presentation, obtained by normalizing the coordinates of the vertices so that they lie on the unit sphere, and identifying the polytope with the convex envelope of the vertices in the Klein model of hyperbolic space. The vertices lie on the sphere at infinity, and correspond to hyperbolic cusps. The vertex figure can be identified to a homothety class of euclidean polytopes, which allows us to define the Coxeter dihedral angles between the facets.
Being a regular polytope, the $24$-cell has an ideal hyperbolic presentation, which we denote $\mathcal{C}$, obtained by normalizing the coordinates of its vertices so that they lie on the unit sphere $S^3\subset \mathbb{R}^4$. By interpreting $S^3$ as the boundary of hyperbolic $4$-space the    vertices of the $24$-cell correspond to  $24$ points at infinity. Their convex envelope in $\mathbb{H}^4$ is the {\em regular ideal hyperbolic $24$-cell}, denoted by $\mathcal{C}$.

The octahedral facets become isometric to regular ideal hyperbolic octahedra.
The cusp section becomes a euclidean cube and therefore all dihedral angles between facets are equal to $\pi/2$, which is a submultiple of $2\pi$. The $24$-cell is the only regular ideal $4$-dimensional polytope which has this property. This allows us to glue isometrically along their facets a finite number copies $\mathcal{C}_1,\dots,\mathcal{C}_n$ of the ideal $24$-cell so that the geometric structures on each $\mathcal{C}_i$ piece  together to produce a hyperbolic structure on the resulting manifold. Following chapter $4$ of \cite{notes}, this is equivalent to verifying that the flat cusp sections piece together to produce closed Euclidean $3$-manifolds.

\section{A minimal volume hyperbolic 4-manifold with two cusps}\label{2cuspidi}
 
In this section, we will construct the hyperbolic $4$-manifold of Theorem \ref{minvol}
We begin considering a $24$-cell $\mathcal{C}$ and we pair each green facet with its opposite using the antipodal map $v\mapsto-v$. We denote with $\mathcal{C}/\sim$ the resulting nonorientable manifold. It is convenient to introduce the following labelings on the strata of the manifold $\mathcal{C}/\sim$:
\begin{enumerate}
\item The manifold $\mathcal{C}/\sim $ has $12$ cusps, in one-to-one with the pairs of opposite vertices of the $24$-cell. Each vertex of the $24$-cell is obtained from a point of the form $(\pm 1,\pm 1,0,0)$ through permutation of the coordinates. It is therefore natural to label each cusp with a $4$-tuple of $+,0,-$ symbols, with two zero entries, defined up to change in signs.
\item There are eight $3$-strata, each corresponding to a pair of opposite {\em red or blue} facets of the $24$-cell $\mathcal{C}$. The red and blue facets lie in hyperplanes of equation $\pm x_1 \pm x_2 \pm x_3 \pm x_4 =2$. We can therefore label each $3$-stratum with a $4$-tuple of $+,-$ symbols,
defined up to change in signs. Each $3$-stratum inherits the coloring of the octahedral facets of the $24$-cell which compose it.
\end{enumerate} 

As a second step, we pair the red facets (those with an even number of minus signs in their labels) through the map $H:\mathbb{R}^4\rightarrow \mathbb{R}^4$ defined by $$H(x,y,z,w)=(-x,-y,z,w).$$ Let's call $\mathcal{A}$ the resulting nonorientable object. 

\begin{prop}
The space $\mathcal{A}$ is a hyperbolic $4$-manifold with eight cusps whose sections fall into two homothety classes: there are four {\em mute} cusps and four {\em non-mute} cusps.
Moreover, it has two disjoint totally geodesic boundary components $X$ and $Y$, each isomorphic to the complement of the minimally twisted $6$-chain link of Figure \ref{6chain}.
\begin{figure}[ht]
\centering
\makebox[\textwidth][c]{
\includegraphics[width=0.39\textwidth]{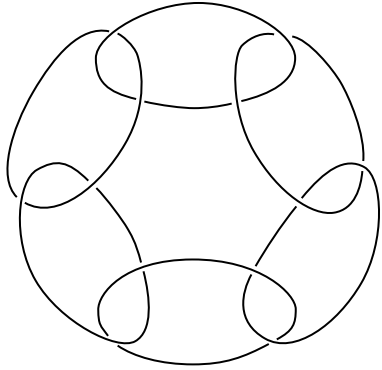}}
\caption{The minimally twisted $6$-chain link. The boundary of $\mathcal{A}$ consists of two disjoint copies of its exterior.}\label{6chain}
\end{figure}
\end{prop}

\begin{proof}
The cusps sections of $\mathcal{A}$ are obtained by pairing with the map $H$ the faces of the cusps sections of the manifold $\mathcal{C}/\sim$, and are represented in Figure \ref{cuspide1}. 
\begin{figure}[htpb]
\centering
\includegraphics{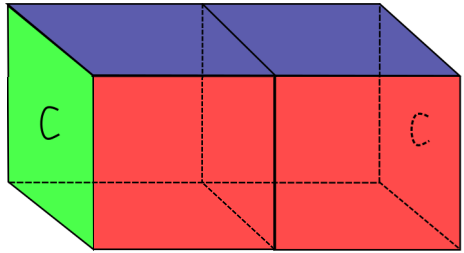}
\caption{A fundamental domain for the cusp section of the manifold $\mathcal{C}/\sim$. The opposite faces of each cube share the same colour, and the green faces are identified in pairs.}\label{cuspide1}
\end{figure} 

There are $12$ such cusps, one for each couple of opposite vertices of the $24$-cell. Each cusp is labeled with a $4$-tuple of $+,-,0$ symbols with two $0$ entries, defined up to change in signs (\emph{i.e.}\ changing every $+$ sign to a $-$ sign and every $-$ sign to a $+$ sign). To mantain a rigorous notation, we will use square parentheses to highlight that we are speaking of equivalence classes.
The cusps labeled $[(+,+,0,0)]$, $[(+,-,0,0)]$, $[(0,0,+,-)]$ and $[(0,0,+,+)]$ are fixed by the map $H$, therefore each of them corresponds to one cusp of $\mathcal{A}$. We label them respectively $m_1$, $n_1$, $m_2$ and $n_2$. %Figure \ref{cuspideidentificata} shows a fundamental domain for their cusp section, together with the identifications on the green and the red faces. 

%Each of these cusps is isomteric to a cusp of the form $Q\times S^1$, where $Q$ is a square of sidelength one, and the $S^1$ factor has length two. A fundamental domain for these cusps is represented in figure \ref{cuspideidentificata}, where the opposite green faces are identified by a translation along the $x$ axis.
%The identifications between the red faces are those obtained by rotating the fundamental domain along the $x$ axis with angle $\pi$, and composing with a translation of length one in the $x$ direction.

We call these the {\em mute} cusps of $\mathcal{A}$, as they have only one boundary component, obtained by glueing together the four blue faces along the edges. This boundary component is isomorphic to a torus, obtained by identifying opposite sides of a square of sidelength two.
 
The mute cusps fall into two isometry classes. The cusps $m_1$ and $n_1$ are orientable, homeomorphic to a twisted $I$-bundle over a Klein bottle $K$. A fundamental domain is represented in Figure \ref{cuspideidentificata}, where a fundamental domain for the Klein bottle $K$ is the mid-section of Figure \ref{cuspideidentificata} parallel to the blue sides, and the boundary of the cusp is isometric to the orientable double cover of $K$.

\begin{figure}[ht]
\centering
\makebox[\textwidth][c]{
\includegraphics[width=0.5\textwidth]{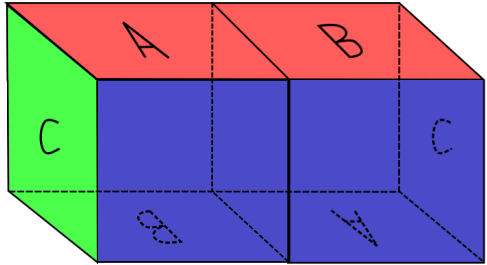}}
\caption{A fundamental domain for the orientable mute cusps of $\mathcal{A}$. The letters show the identifications between green and red faces.}\label{cuspideidentificata}
\end{figure}

The mute cusps $m_2$ and $n_2$ are non-orientable, homeomorphic to the product $M\times S^1$, where $M$ is a M\"obius band. A fundamental domain is represented in Figure \ref{cuspideidentificata2}, where a fundamental domain for the M\"obius band is the front green face labelled by ``$C$". The boundary of the cusp is isometric to $\partial{M}\times S^1$.

\begin{figure}[ht]
\centering
\makebox[\textwidth][c]{
\includegraphics[width=0.5\textwidth]{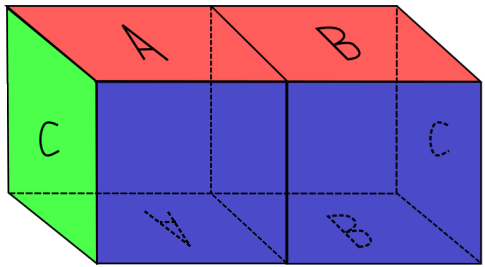}}
\caption{A fundamental domain for the non-orientable mute cusps of $\mathcal{A}$. The letters show the identifications between green and red faces.}\label{cuspideidentificata2}
\end{figure}

The other four cusps, labeled $a$, $b$, $c$ and $d$, are obtained from the following pairings between the cusps of $\mathcal{C}/\sim$, induced by the map $H$:
\begin{center}
$a:[(+,0,+,0)]\leftrightarrow[(-,0,+,0)]$,

$d:[(0,+,0,+)]\leftrightarrow[(0,-,0,+)]$,

$b:[(+,0,0,+)]\leftrightarrow[(-,0,0,+)]$,

$c:[(0,+,+,0)]\leftrightarrow[(0,-,+,0)].$
\end{center}

The resulting cusp sections are all isometric to $T\times I$, where $T$ is a torus obtained by identifying opposite sides of a square of sidelength two, and $I$ is a unit interval. A fundamental is represented in Figure \ref{cuspidedoppia}.
\begin{figure}
\centering
\makebox[\textwidth][c]{
\includegraphics[width=0.5\textwidth]{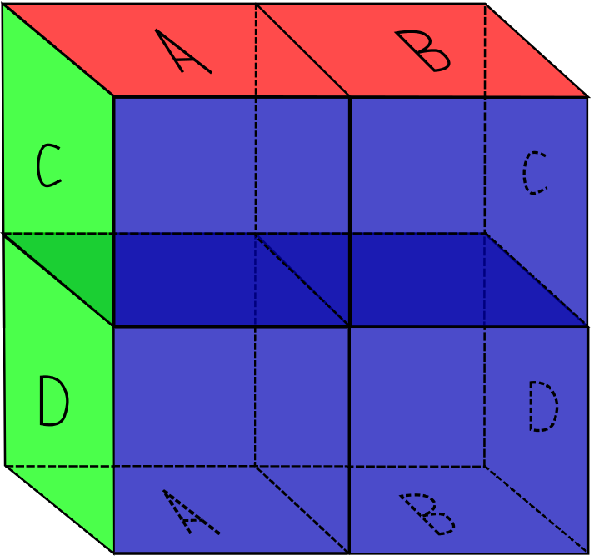}}
\caption{A fundamental domain for the non-mute cusps of $\mathcal{A}$. Letters show the identifications between green and red faces.}\label{cuspidedoppia}
\end{figure}
These cusps have two boundary components, once again isomorphic to a square torus of sidelength two.
In all the above cases, the cusp sections are Euclidean $3$-manifolds. Therefore $\mathcal{A}$ possesses a hyperbolic structure with totally geodesic boundary.

Each of the boundary components of $\mathcal{A}$ is obtained by glueing together along their faces four of the blue octahedral facets of the $24$-cell $\mathcal{C}/\sim$, and it can be described as the result of a double mirroring of a regular ideal hyperbolic octahedron, as in Figure \ref{bordi} (top), where the first mirroring is performed along the green faces and the second along the red. The mirroring on the green faces is the result of the pairing of the green octahedral facets of the $24$-cell, which is performed with the antipodal map $v\mapsto -v$. The mirroring on the red faces is the result of the pairing of the red facets, which is performed by the affine map $H$.

\begin{figure}
\centering
\makebox[\textwidth][c]{
\includegraphics[width=0.9\textwidth]{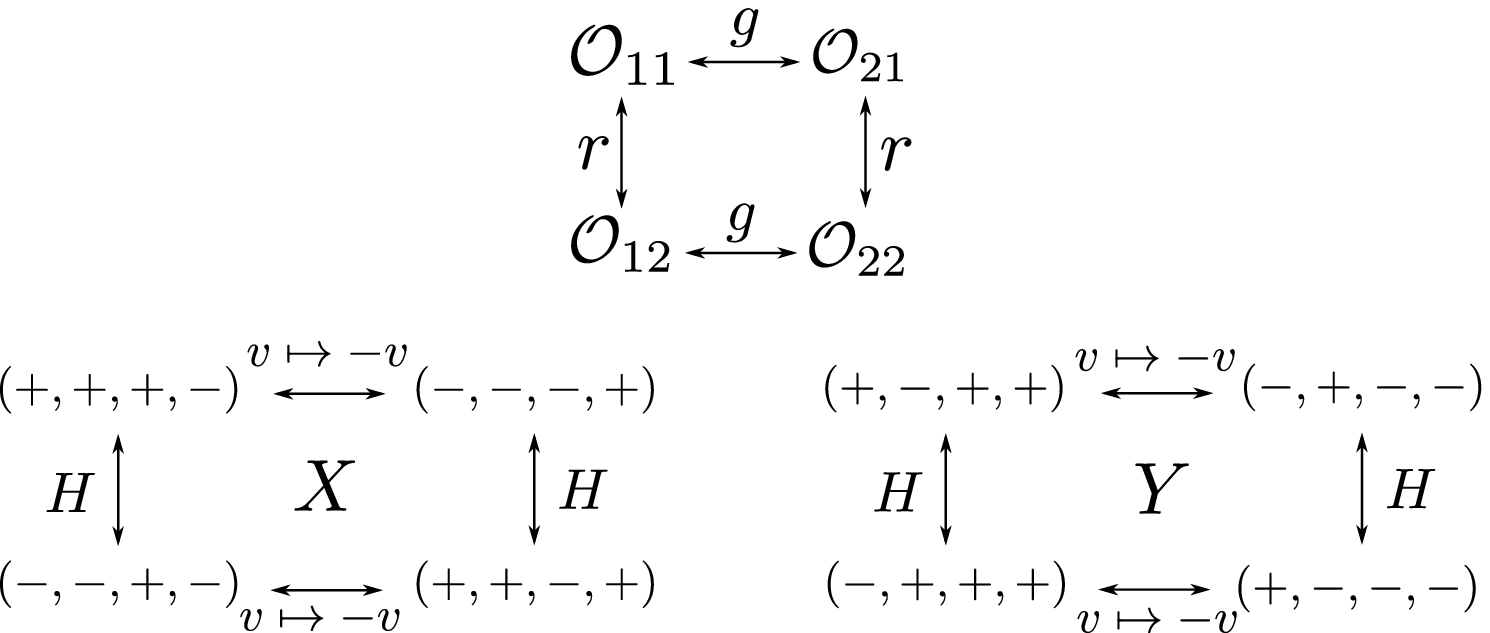}}
\caption{Presentations of the boundary components of $\mathcal{A}$ as a double mirroring of a regular ideal hyperbolic octahedron. Labels show which $3$-strata of the $24$-cell compose each boundary component. The antipodal map identifies the green faces, while $H$ identifies the red ones.}\label{bordi}
\end{figure}
As proven in \cite{martelli}, these boundary components are isomorphic to the exterior of the minimally twisted $6$-chain link, represented in Figure \ref{6chain}.
\end{proof}

The exterior $M_L$ of the minimally twisted $6$-chain link represented in Figure \ref{6chain} has two isometries $W$ and $V$, obtained by reflecting respectively in the green and the red faces. The isometry $W$ (resp.\ $V$) acts on the square diagram of Figure \ref{bordi} (top) interchanging $\mathcal{O}_{1i}$ with $\mathcal{O}_{2i}$ (resp.\ $\mathcal{O}_{i1}$ with $\mathcal{O}_{i2}$). In the case of the boundary components $X$ and $Y$ of $\mathcal{A}$, the isometries $W$ and $V$ are induced respectively by the antipodal map $v\mapsto -v$ and the affine map $E$ on $\mathcal{C}$.
\begin{prop}\label{isometrie}
Every isometry of $M_L$ preserves its decomposition as a union of four copies of the ideal hyperbolic octahedron $\mathcal{O}$. There is an exact sequence $$0\rightarrow \mathbb{Z}_2 \oplus \mathbb{Z}_2\rightarrow \text{\rm Isom}(M_L)\rightarrow \text{\rm Isom}(\mathcal{O})\rightarrow 0.$$
Where $\mathbb{Z}_2 \oplus \mathbb{Z}_2$ is the group generated by $W$ and $V$.
\end{prop}

\begin{proof}
The four octahedra $O_{ij}$ represent the canonical Epstein-Penner decomposition of $M_L$, and are therefore preserved by any isometry.
Every isometry of $M_L$ is the composition of an isometry in $\mathbb{Z}_2 \oplus \mathbb{Z}_2$ with one which fixes the octahedron $\mathcal{O}_{11}$. This defines the map onto $\text{\rm Isom}(\mathcal{O}_{11})\cong \text{\rm Isom}(\mathcal{O})$. Its kernel is precisely the group generated by $W$ and $V$.
\end{proof}

\subsection{Gluing of the boundary components}
As a last step to obtain a manifold with no boundary, we need to glue the boundary components $X$ and $Y$ of $\mathcal{A}$ together using an appropriate isometry $\phi:X\rightarrow Y$. As a consequence of Proposition \ref{isometrie}, to encode $\phi$ we need to specify the choice of two octahedra $\mathcal{O}_1$ and $\mathcal{O}_2$ from $X$ and $Y$ respectively, together with an isometry from $\mathcal{O}_1$ to $\mathcal{O}_2$.
We choose $\mathcal{O}_1$ as the octahedron with labeling $(+,+,+,-)$ and $\mathcal{O}_2$ as that with labeling $(+,-,+,+)$. This choice is purely arbitrary and, as a consequence of Proposition \ref{isometrie}, does not affect the argument which will follow.

%Let's call $m_1$ and $m_2$ the mute cusps and $a$, $b$, $c$, $d$ those with two boundary components. 

Each of the octahedra $\mathcal{O}_i$ has two opposite ideal vertices which correpond to the mute cusps of $\mathcal{A}$ and the remaining four corresponding to the cusps labeled $a$, $b$, $c$ and $d$ as in Figure \ref{ottaedro}. Notice that the boundary components of the mute cusps labeled $m_i$ belong to $X$, while those labeled $n_i$ belong to $Y$, therefore both $X$ and $Y$ have one cusp corresponding to a boundary component of an orientable mute cusp, and another cusp correponding to a boundary component of a non-orientable mute cusp of $\mathcal{A}$. Each non-mute cusp has one boundary component on $X$ and one on $Y$. Furthermore the cusps labeled $a$ and $d$ always lie on opposite vertices of the octahedra, as do those labeled $b$ and $c$.
\begin{figure}[ht]
\centering
\makebox[\textwidth][c]{
\includegraphics[width=0.6\textwidth]{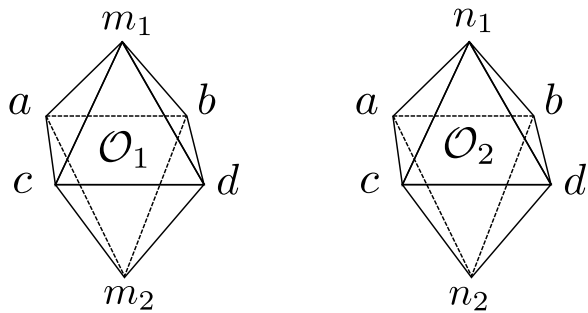}}
\caption{Model for the octahedra $O_i$. Each vertex is labeled with the cusp to which it belongs.}\label{ottaedro}
\end{figure}

To define $\phi$, we choose the map from $\mathcal{O}_1$ to $\mathcal{O}_2$ which pairs the vertices in the following way:
\begin{center}
$\mathcal{O}_1\xrightarrow{\phi} \mathcal{O}_2$

$m_1\mapsto b$

$m_2\mapsto c$

$b \mapsto d$

$c \mapsto a$

$a \mapsto n_1$

$d \mapsto n_2.$
\end{center}

\begin{prop}
Let $\mathcal{G}$ be the manifold $$\mathcal{A}/(x\sim \phi(x))$$ obtained by gluing the boundary components $X$ and $Y$ together using $\phi$ as defined above. The manifold $\mathcal{G}$ has a hyperbolic structure with two cusps. Its hyperbolic volume is equal to $v_m=4\pi^2/3$.
\end{prop}

\begin{proof}
The hyperbolic structure of $\mathcal{G}$ is inherited from that of $\mathcal{A}$, since we are glueing by an isometry its two disjoint totally geodesic boundary components. The effect of this glueing is to concatenate the cusps of $\mathcal{A}$ along their boundary components in two groups of four, with the mute cusps at the ends and the non-mute cusps in the middle, as in Figure \ref{duecuspidi}.

\begin{figure}[ht]
\centering
\makebox[\textwidth][c]{
\includegraphics[width=0.4\textwidth]{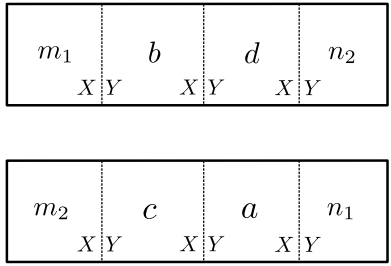}}
\caption{Cusps of $\mathcal{G}$ as a result of glueing the cusps of $\mathcal{A}$. The letters in the middle of the squares indicate the cusps of $\mathcal{A}$. Dotted lines indicate their boundary components.}\label{duecuspidi}
\end{figure}
Since $\mathcal{G}$ is obtained by pairing the facets of one regular ideal hyperbolic $24$-cell, it has exactly the same volume $v_m=4\pi^2/3$.
\end{proof}

\begin{oss}
The homeomorphism class of a closed Euclidean $3$-manifold is determined by its orientabilty and by its first homology group, as shown in \cite[pag.\ 122]{wolf}. 
The cusp sections of $\mathcal{G}$ are obtained by identifying the orientable mute cusp $A=K\simtimes I$ with the non-orientable mute cusp $B=M\times S^1$ along their boundary tori. By computing the glueing map and applying the Mayer-Vietoris sequence to the pair $(A,B)$, it is easy to see that
the first homology group of both cusp sections is isomorphic to $\mathbb{Z}\times \mathbb{Z}/4\mathbb{Z}$, therefore the cusps are of topological type $\mathcal{B}_4$. 

In the orientable double cover $\widetilde{\mathcal{G}}$ of $\mathcal{G}$, the cusps sections of $\mathcal{G}$ lift to their orientable double covers, which is obtained by mirroring the orientable mute cusp in its boundary. The first homology group of the resulting manifold is isomorphic to $\mathbb{Z}\times \mathbb{Z}/2\mathbb{Z}\times \mathbb{Z}/2\mathbb{Z}$, so these cusp sections are of topological type $\mathcal{G}_2$.
\end{oss}

\section{A one-cusped hyperbolic 4-manifold}\label{1cuspide}
We now proceed to construct the hyperbolic $4$-manifold of Theorem \ref{2minvol}.
We begin by considering two copies $\mathcal{C}_1$ and $\mathcal{C}_2$ of the regular ideal hyperbolic $24$-cell $\mathcal{C}$, with opposite orientations, and we glue them together along their {\em green} facets via the map induced by the identity on $\mathcal{C}$. 
We call the resulting manifold the {\em mirrored $24$-cell}, and denote it by $\mathcal{S}=\mathcal{C}_1\cup\mathcal{C}_2$.
We subsequently glue each red facet of $\mathcal{C}_i$ to the opposite face of $\mathcal{C}_i$ using the antipodal map $v\mapsto -v$.

Let us call $\mathcal{D}$ the resulting nonorientable space. 
\begin{prop}
The space $\mathcal{D}$ is a hyperbolic manifold with twelve cusps whose sections are all isometric to each other. 
It has four disjoint totally geodesic boundary components, each isomorphic to the exterior of the minimally twisted $6$-chain link of Figure \ref{6chain}.
\end{prop}
 
\begin{proof}
The cusp sections of $\mathcal{D}$ are obtained by glueing along their red boundary faces the cusp sections of the mirrored $24$-cell $\mathcal{S}$. The glueing is induced by the antipodal map on each $\mathcal{C}_i$. The resulting cusp section is represented in figure \ref{cuspidedoppia2}.
\begin{figure}
\centering
\makebox[\textwidth][c]{
\includegraphics[width=0.5\textwidth]{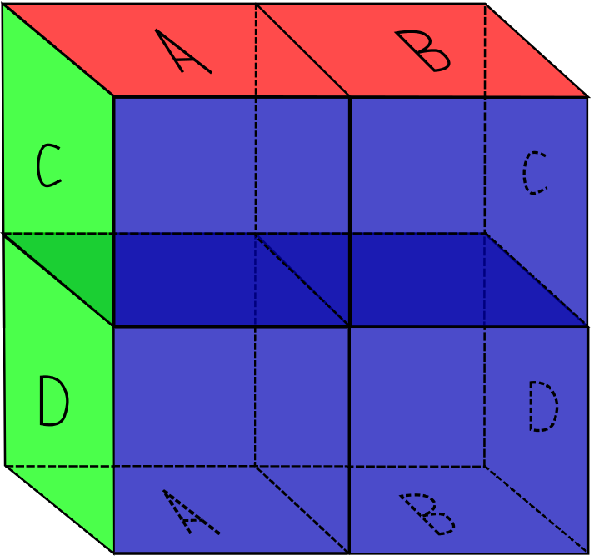}}
\caption{A fundamental domain for the cusp sections of $\mathcal{D}$. Letters show the identifications between green and red faces.}\label{cuspidedoppia2}
\end{figure}

There are  $12$ such cusps, one for each pair of opposite vertices of the $24$-cell, naturally labeled by a $4$-tuple of $0,+,-$ symbols with two zero entries, defined modulo changes in signs (changing every $+$ sign to a $-$ sign and every $-$ sign to a $+$ sign). To simplify notations we will use the following labelings:
\begin{center}
$a_1=[(+,+,0,0)]$, $a_2=[(+,-,0,0)]$,

$d_1=[(0,0,+,+)]$, $d_2=[(0,0,+,-)]$,

$b_1=[(+,0,+,0)]$, $b_2=[(+,0,-,0)]$,

$e_1=[(0,+,0,+)]$, $e_2=[(0,+,0,-)]$,

$c_1=[(+,0,0,+)]$, $c_2=[(+,0,0,-)]$,

$f_1=[(0,+,+,0)]$, $f_2=[(0,1,-1,0)]$.
\end{center}
The cusps sections are Euclidean $3$-manifolds, isometric to a product $T\times I$, where $T$ is a square torus of sidelength two and $I$ is a unit interval. Therefore $\mathcal{A}$ possesses a hyperbolic structure with totally geodesic boundary.

Each boundary $3$-stratum is obtained from the gluings of four of the blue octahedral facets of the $24$-cells $\mathcal{C}_1$ and $\mathcal{C}_2$ along their triangular $2$-faces, and is representable as a double mirroring of an octahedron with a red/green checkerboard coloring as in Figure \ref{quattrobordi}. 
Indeed, the mirroring on the green faces glues a blue octahedral facet $O_1$ of $\mathcal{C}_1$ to the correponding facet $O_2$ of $\mathcal{C}_2$,
while the mirroring on the red faces glues a blue octahedral facet $O_i$ to its opposite $-O_i$ in $\mathcal{C}_i$.
These boundary components are naturally labeled as the {\em blue} facets of the $24$-cell, \emph{i.e.}\ with a $4$-uple of $+,-$ symbols, defined up to change in signs (changing every $+$ sign to a $-$ sign and every $-$ sign to a $+$ sign). To simplify notations, we will use the following labelings:
\begin{center}
$B_1=[(-,+,+,+)]$, $B_2=[(+,-,+,+)]$,

$B_3=[(+,+,-,+)]$, $B_4=[(+,+,+,-)]$.
\end{center}

\begin{figure}[ht]
\centering
\makebox[\textwidth][c]{
\includegraphics[width=0.22\textwidth]{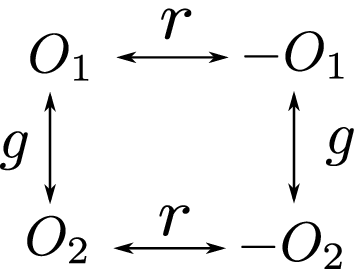}}
\caption{Presentation of a boundary component of $\mathcal{D}$. It is made up of blue facets $O_1$ and $O_2$ from $\mathcal{C}_1$ and $\mathcal{C}_2$, both corresponding to the same facet $O$ of $\mathcal{C}$, together with their opposites $-O_1$ and $-O_2$.}\label{quattrobordi}
\end{figure}
\end{proof}

\subsection{Glueing of the boundary components}
We now need to identify in pairs the boundary components $B_i$, for $i=1,\dots,4$ to produce a manifold with no boundary.
We choose to pair $B_3$ with $B_1$ via an isometry $\phi_1$ and $B_4$ with $B_2$ via an isometry $\phi_2$. As in the previous section, to define the maps $\phi_i$, $i=1,2$, we must choose one octahedron from the domain and one from the codomain, and then specify how we pair them. As a consequence of Proposition \ref{isometrie}, the first choice does not affect the argument which will follow. We can choose, for example, for $\phi_1$ (resp.\ $\phi_2$) to pair the octahedral facet of $\mathcal{C}_1$ labeled $(-,+,+,+)$ (resp. $(+,-,+,+)$) with the facet of $\mathcal{C}_1$ labeled $(+,+,-,+)$ (resp.\ $(+,+,+,-)$).
These octahedra have ideal vertices belonging to different cusps of $\mathcal{D}$ as shown in Figure \ref{4ottaedri}. 

\begin{figure}[ht]
\centering
\makebox[\textwidth][c]{
\includegraphics[width=0.5\textwidth]{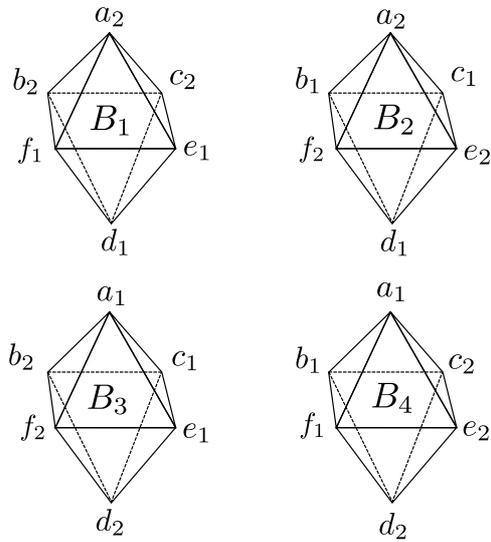}}
\caption{Correspondence between cusps of $\mathcal{D}$ and ideal vertices of the octahedra which constitute the boundary components $B_i$.}\label{4ottaedri}
\end{figure}

To choose the glueing maps $\phi_1$ and $\phi_2$, we choose the following pairings between the vertices:
\begin{center}
$B_4\xrightarrow{\phi_2} B_2\;\;B_3\xrightarrow{\phi_1} B_1$

$a_1\mapsto c_1\;\;\;\;a_1\mapsto b_2$

$d_2\mapsto f_2\;\;\;\;d_2\mapsto e_1$

$c_2\mapsto e_2\;\;\;\;b_2\mapsto f_1$

$f_1\mapsto b_1\;\;\;\;e_1\mapsto c_2$

$b_1\mapsto a_2\;\;\;\;f_2\mapsto a_2$

$e_2\mapsto d_1\;\;\;\;c_1\mapsto d_1$
\end{center}

\begin{prop}\label{acca}
Let $\mathcal{H}$ be the non-orientable manifold $$\mathcal{D}/(x\sim\phi_i(x))$$ obtained by glueing in pairs the boundary components $B_i$ using $\phi_1$ and $\phi_2$ as defined above. $\mathcal{H}$ has a hyperbolic structure with total volume equal to $2\cdot v_m$ and one toric cusp.
\end{prop}

\begin{proof}
The hyperbolic structure of $\mathcal{D}$ induces one on $\mathcal{H}$. The $12$ cusp sections of $\mathcal{D}$ are concatenated along their boundary components to form the only cusp of $\mathcal{H}$ as in Figure \ref{trenino}.

\begin{figure}[ht]
\centering
\makebox[\textwidth][c]{
\includegraphics[width=\textwidth]{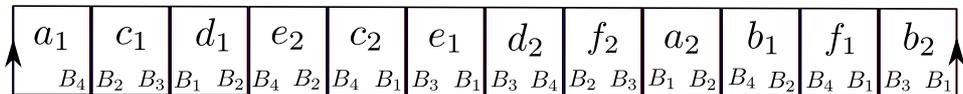}}
\caption{The cusp of $\mathcal{H}$ as concatenation of the cusps of $\mathcal{D}$. Vertical edges mark the identifications induced on the boundary components $B_i$ by the maps $\phi_1$ and $\phi_2$. The cusp shape is that of a torus obtained by identifying opposite faces of a $2\times2\times 12$ rectangular parallelepiped.}\label{trenino}
\end{figure}
The volume is $2\cdot v_m$, because $\mathcal{H}$ is obtained by pairing the facets of two copies of the regular ideal hyperbolic $24$-cell.
\end{proof}

%\begin{oss}
%The orientable double cover of the manifold $\mathcal{H}$ of Proposition \ref{acca} is the example of one cusped hyperbolic $4$-manifold constructed by Kolpakov and Martelli in \cite{martelli}.
%\end{oss}

\end{document}